\documentclass[12pt]{siamart251216}
\usepackage[english]{babel}
\usepackage{blindtext}

\usepackage[utf8]{inputenc}
\usepackage{amsmath,amssymb,amsfonts}
\usepackage{bm}
\usepackage{enumitem}
\usepackage{algorithm}
\usepackage{algpseudocodex}
\usepackage{float}
\usepackage{blkarray}
\usepackage{tikz-cd}
\usepackage{tcolorbox}
\newtcolorbox{mybox}{colback=red!5!white,colframe=red!75!black, sharp corners = all}
\usepackage{bclogo} 
\def\ucsign{\bcpanchant}

\setenumerate{labelindent=0.25\parindent,leftmargin=*}

\DeclareMathOperator{\pf}{pf}

\DeclareMathOperator{\diag}{diag}

\def\F{{\mathsf R}}

\def\M{{\mathcal M}}
\def\N{{\mathbb N}}

\def\D{{\mathbf D}}

\def\SM{\raisebox{.0pt}{-}}
\let\det\undefined
\DeclareMathOperator{\det}{det}

\DeclareMathOperator{\sgn}{sgn}

\DeclareMathOperator{\SGN}{\mathsf S}

\newenvironment{psmallmatrix}
  {\left(\begin{smallmatrix}}
  {\end{smallmatrix}\right)}

\theoremstyle{plain}

\theoremstyle{definition}
\newtheorem{example}[theorem]{Example}

\newtheorem{remark}[theorem]{Remark}

\numberwithin{equation}{section}
\numberwithin{algorithm}{section}
\numberwithin{table}{section}

\headers{On the computations of the Cullis' determinant}{A. Guterman and A. Yurkov}

\title{On the computations of the Cullis' determinant}

\author{Alexander Guterman\thanks{Department of Mathematics, Bar Ilan University, Ramat-Gan, 5290002, Israel.}
\and Andrey Yurkov\thanks{\emph{Correspnding author}, Department of Mathematics, Bar Ilan University, Ramat-Gan, 5290002, Israel
  (\email{andrey.yurkov@biu.ac.il}).}}

\ifpdf
\hypersetup{
  pdftitle={On the computations of the Cullis' determinant},
  pdfauthor={A. Guterman and A. Yurkov}
}
\fi

\begin{document}

\maketitle

\begin{abstract}
The Cullis' determinant is a generalization of the ordinary determinant for rectangular matrices. It is defined as the alternating sum of maximal minors of a given matrix. In this paper we express the Cullis' determinant of a matrix $X$  as the Pfaffian of the matrix obtained from $X$ by matrix multiplication and transposition. 

Relying on this result, we present an efficient polynomial-time division-free algorithm for calculating the Cullis' determinant of a given matrix with entries belonging to the commutative ring. We provide an asymptotical analysis of its arithmetical complexity in comparison to the definition-based algorithm. 

In addition, we derive formulas for horizontal expansion of the Cullis' determinant which complements the existing formula for Laplace expansion along the columns of a matrix.
\end{abstract}


%
%



\section{Introduction}
\label{sec:intro}

The Cullis' determinant is a generalization of the determinant of square matrices for rectangular matrices. It was introduced by Cullis in 1913 in his monograph ``Matrices and Determinoids'' (see~\cite{cullis1913}) as follows:

\begin{definition}[{Cf.~\cite[\S3, p. 12 and \S30, Theorem~I.]{cullis1913}}]\label{def:CullisDet}
Let  $n \ge k\ge 1$ be integers, $X = (x_{i\,j})$ be an $n\times k$ matrix with entries belonging to the commutative ring $\F$. Then the Cullis' determinant $\det_{n\, k}(X)$ of $X$ is defined by the formula
\begin{equation*}
\det_{n\,k}(X) = (-1)^{(1 + \ldots + k)}\sum_{1 \le c_1 < \ldots < c_k \le n } (-1)^{c_1 + \ldots + c_k} \begin{vmatrix}x_{c_1\, 1} & \cdots & x_{c_1\, k}\\ \vdots & \ddots & \vdots\\ x_{c_k\, 1} & \cdots & x_{c_k\, k}\end{vmatrix}.
\end{equation*}
That is, the Cullis' determinant of $X$ is an alternating sum of basic minors of $X$. We also use the following notation for $\det_{n\, k} (X)$:
\[
\det_{n\,k}(X)=\begin{vmatrix}x_{1\,1} & \cdots & x_{1\,k}\\
\vdots & \cdots & \vdots\\
x_{n\,1} & \cdots & x_{n\,k}
\end{vmatrix}_{n\,k}.
\]
When $n = k$, $\det_{n\,n}$ is also denoted as $\det_{n}$ and is clearly equal to the ordinary determinant of a square matrix. In this case, we also use the notion $\det$ (omitting subscripts).

For $k > n$ the Cullis' determinant of an $n\times k$ matrix $X$ is defined by $\det_{n\, k}(X) = \det_{k\,n}(X^t)$ (see~\cite[Definition~1]{radic1966}). 
\end{definition}

The Cullis' determinant has many properties similar to those of the ordinary determinant (see \cite[\textsection 5, \textsection 27, \textsection 32]{cullis1913}, \cite{NAKAGAMI2007422} and~\cite{radic1966} for detailed proofs).
\emph{\begin{enumerate}[label=\textsc{C\arabic*.}, ref=\textsc{C\arabic*}]
\item For an $n\times n$ matrix $X$, $\det_{n\,n}(X) = \det (X).$
\item $\det_{n\,k}(X)$ is a linear function of columns of $X$.
\item If a matrix $X$ has two identical columns or one of its columns is a linear combination of other columns, then $\det_{n\,k}(X)$ is equal to zero.
\item Interchanging any two columns of $X$ changes the sign of $\det_{n\,k}(X)$.
\item Adding a linear combination of the columns of $X$ to another column of $X$ does not change $\det_{n\,k}(X)$.
\item\label{CD:Laplace} $\det_{n\,k}(X)$ can be calculated using the Laplace expansion along a column of $X$. That is, if  $1 < k \le n$, then for any $n \times k$ matrix $X = (x_{i\, j})$ the expansion of $\det_{n\, k}(X)$ along the $j$-th column is given by
\[
\det_{n\, k}(X) = \sum_{i = 1}^n (-1)^{i+j} x_{i\, j} \det_{(n-1)\,(k-1)}\Bigl(X(i|j)\Bigr),
\]
where $X(i|j)$ denotes the $(n-1) \times (k-1)$ matrix obtained from $X$ by deleting the $i$-th row and the $j$-th column.
\end{enumerate}}

After the work of Cullis, $\det_{n\,k}$ was studied by Radi\'{c} in~\cite{radic1966,radic2005,radic1991,radic2008},\linebreak Makarewicz, Mozgawa, Pikuta and Sza\l{}kowski in~\cite{makarewicz2014,makarewicz2016,makarewicz2020}, Amiri, Fathy and Bayat in~\cite{amiri2010}, Nakagami and Yanai in~\cite{NAKAGAMI2007422} and by the authors in~\cite{Guterman2024,Guterman2025,Guterman2025b,Guterman2025c,Guterman2025l}.  

While there exist polynomial-time algorithms for calculating the determinant of a square matrix such as Gaussian elimination, every algorithm developed so far for finding the value of the Cullis determinant of a given matrix requires exponential time. In this paper we introduce a formula which allows us to develop a polynomial-time algorithm for computing the Cullis' determinant (see the equality~\eqref{thm:DetEqPfAll:eqq1}). This formula involves the notation provided in the following several definitions. 

\begin{definition}\label{def:NAKAG1}
By $[n]$ we denote the set $\{1, \ldots, n\}$.
\end{definition}

\begin{definition}Let $n \ge 1$ be an integer. By $S_n \subseteq [n]^n$ we denote the set of all permutations of $[n]$, i.e., the set of all sequences $(i_1, \ldots, i_n)$ such that $i_k \neq i_l$ for all $1 \le k, l \le n$. If $\pi = (i_1, \ldots, i_n) \in S_{n}$ and $1 \le k \le n$, then $\pi(k)$ denotes $i_k$.
\end{definition}

\begin{definition}\label{def:SignOfSec}Assume that $n \ge 2$. Given a tuple  $(c_1, \ldots, c_n)$ of natural numbers, the sign of  $(c_1, \ldots, c_n)$ is denoted by $\sgn (c_1, \ldots, c_n)$ and is defined by
\[
\sgn (c_1, \ldots, c_n) =
\begin{cases}
0, & \exists i, j\in [n]\colon\, i \neq j\;\mbox{and}\; c_i = c_j\\
(-1)^{|\{(i, j) | 1 \le i < j \le n, c_i > c_j\}|}, & \mbox{otherwise}
\end{cases}.
\]
If $\pi \in S_n$ is a permutation, then $\sgn (\pi(1), \ldots, \pi(n))$ is also denoted by $\sgn(\pi).$
\end{definition}

\begin{definition}By $\Pi_{2m} \subseteq S_{2m}$ we denote the set of all permutations $\pi$ such that\linebreak  $\pi(1) < \pi(3) \ldots < \pi(2m-1)$ and  $\pi(2l - 1) < \pi(2l)$ for all $1 \le l \le m$.
\end{definition}

\begin{definition}[{\cite[Cf.~p.~103]{STEMBRIDGE199096}}]\label{def:pfaffian}Let $X$ be a $2m\times 2m$ skew-symmetric matrix. Then the Pfaffian $\pf(X)$ of $X$ is defined to be the function: 
\[
\pf(X) = \sum_{\pi \in \Pi_{2m}} \sgn (\pi) \Pi_{l = 1}^{m} x_{\pi(2l-1)\, \pi(2l)}.
\]
\end{definition}

The Pfaffian is a polynomial function of the entries of antisymmetric matrix $A$ which satisfies the equality $\pf^2(A) = \det(A)$. The proof of this equality and the further information on the Pfaffian could be found, for example, in~\cite[III.5]{artin2016geometric}. 
%

\begin{definition}\label{def:Dmat}
By $\D^{(n)} = (d^{(n)}_{i\,j})$ we denote the skew-symmetric $n\times n$ matrix  defined by 
$$d^{(n)}_{i\,j} = \sgn(i,j) \cdot (-1)^{i+j},\; 1 \le i, j \le n.$$
\end{definition}


Using these definitions, the discussed equality is formulated as follows:

\begin{theorem}\label{thm:DetEqPfAll}Let $n, k \ge 1$, $X = (x_{i\,j})$ a matrix of size $n\times k$ with entries from the commutative ring $\F$. If $n \ge k$, then the following equality in  $\F[x_{1\,1}, \ldots, x_{n\,k}]$ holds
\begin{equation}\label{thm:DetEqPfAll:eqq1}
\det_{n\,k}(X) = 
\begin{cases}
(-1)^{1 + 2 + \ldots + k}\pf \left(X^t \D^{(n)} X\right), & k\;\;\mbox{is even},\\[15pt]
(-1)^{1 + 2 + \ldots + (k+1)}\pf \left((X')^t \D^{(n+1)} X'\right), & k\;\;\mbox{is odd},
\end{cases}
\end{equation}
where 
 $X'$ is a $(n+1)\times (k+1)$ matrix defined by  
 $$X' =  \left(
\begin{array}{c|ccc}
\multicolumn{1}{c}{1} & 0 & \cdots & 0 \\
\cline{2-4}
0      & &   & \\
\vdots & & X & \\
0      & &   & 
\end{array}
\right).$$
If $n < k$, then the equality~\eqref{thm:DetEqPfAll:eqq1} holds for $\det_{k\,n}(X^t)$ which is equal to $\det_{n\,k}(X)$ by the definition.
\end{theorem}

This theorem could be considered as a particular case of the following minor summation formula due to Ishikawa and Wakayama~\cite{Ishikawa01081995}.

\begin{theorem}[{Cf.~\cite[Theorem~1]{Ishikawa01081995}}]\label{thm:OkadaPf}Let $m, n$ be positive integers such that $m \le n$  and $T = (t_{i\,k})$ is an $m\times n$ matrix with entries belonging to the fixed commutative ring.
\begin{enumerate}[label=(\alph*), ref=(\alph*)]
\item\label{thm:OkadaPf:part1} Assume that $m$ is even. If $B = (b_{i\,k})$ is any $n\times n$ skew-symmetric matrix, then
\begin{multline*}
\sum_{1 \le c_1 < \ldots < c_m \le n} \pf(B[c_1,\ldots, c_m|c_1,\ldots, c_m]) \det (T(|c_1,\ldots, c_m]) = \pf (Q)\\
\mbox{as elements of}\;\;\F[t_{1\,1}, \ldots, t_{m\,n}],
\end{multline*}
where $Q$ is the skew-symmetric $m\times m$ matrix defined by $Q = T B T^t$.
\item\label{thm:OkadaPf:part2} Assume that $m$ is odd. If $B = (b_{i\,k})$ is any skew-symmetric $(n+1)\times (n+1)$ matrix, then
\begin{multline*}
\sum_{1 \le c_1 < \ldots < c_m \le n}\bigg(\pf(B[1, (c_1 + 1),\ldots, (c_m + 1)|1, (c_1 + 1),\ldots, (c_m + 1)])\\
\times \det (T(|c_1,\ldots, c_m])\bigg)
 = \pf (Q')
\end{multline*}
as elements of $\F[t_{1\,1}, \ldots, t_{m\,n}]$, where $Q' = T' B (T')^t$ and
 $$T' =  \left(
\begin{array}{c|ccc}
\multicolumn{1}{c}{1} & 0 & \cdots & 0 \\
\cline{2-4}
0      & &   & \\
\vdots & & T & \\
0      & &   & 
\end{array} 
\right).$$
\end{enumerate}
\end{theorem}
%

We also provide a straightforward and elementary proof of Theorem~\ref{thm:DetEqPfAll} which is carried out by comparing the coefficients of the monomials on the both sides of~\eqref{thm:DetEqPfAll:eqq1} (see Section~\ref{sec:DetEqPfAllAltProof}).

Thus, the formula~\eqref{thm:DetEqPfAll:eqq1} reduces the computation of the Cullis' determinant to the computation the Pfaffian. Relying on the recently developed polynomial division-free algorithm for computing the Pfaffian~\cite{PRZEZDZIECKI2025106550}, we present an effective polynomial-time algorithm for computing the Cullis' determinant of a given matrix with entries belonging to the commutative ring (Algorithm~\ref{alg:DETNKALG}) and show that its resulting arithmetical complexity is substantially lower than the complexity of the definition-based algorithm.


%
%

In addition to the results presented above, we obtain the following consequence of Theorem~\ref{thm:DetEqPfAll}, being a row analogue of the Laplace expansion formula for the Cullis' determinant along the columns of a matrix. 

\begin{theorem}\label{thm:HorExpEven}If $n \ge 2m$, $1 \le j \le 2m$ and $X$ is an $n \times 2m$ matrix with entries from the commutative ring $\F$, then
\begin{equation*}
\det_{n\,2m}(X) = \sum_{1 \le i \neq j \le 2m} (-1)^{i + j - 1} \det_{n\,2}\bigl(X(|i,j]\bigr)\det_{n\,(2m-2)}\bigl(X(|i,j)\bigr),
\end{equation*}
where $X(|i,j]$ denotes an $n\times 2$ matrix consisting of the $i$-th and the $j$-th columns of $X$ and $X(| i, j)$ denotes an $n\times (2m-2)$ matrix obtained from $X$ by striking out its $i$-th and the $j$-th columns.  
\end{theorem}

\begin{theorem}\label{thm:HorExpOdd}If $n \ge 2m-1$ and $X$ is an $n \times (2m-1)$ matrix with entries from the commutative ring $\F$, then
\begin{equation*}
\det_{n\, (2m-1)}(X) = \sum_{1 \le i \le 2m-1} (-1)^{i - 1} \det_{n\,1}\left(X(|i]\right)\det_{n\,(2m-2)}\left(X(| i)\right),
\end{equation*}
where $X(|i]$ denotes an $n\times 1$ matrix consisting of the $i$-th  column of $X$ and $X(| i)$ denotes an $n\times (2m-2)$ matrix obtained from $X$ by striking out its $i$-th column. 
\end{theorem}

%

The further text is organized as follows: in Section~\ref{sec:notation} we introduce the necessary notation, definitions and preliminary facts regarding the Cullis' determinant and the Pfaffian; in Section~\ref{sec:main} we explain how the formula~\eqref{thm:DetEqPfAll:eqq1} follows from Ishikawa-Wakayama minor summation formula; in Section~\ref{sec:computation} we present an efficient polynomial-time algorithm for finding $\det_{n\,k}(X)$ using the formula~\eqref{thm:DetEqPfAll:eqq1} (Algorithm~\ref{alg:DETNKALG}) and analyze its computational complexity in comparison to the definition-based algorithm; Section~\ref{sec:horizexp} consists of the proofs of Theorem~\ref{thm:HorExpEven} and Theorem~\ref{thm:HorExpOdd} on the horizontal expansion of the Cullis' determinant; in the appendix we provide an alternative elementary proof of Theorem~\ref{thm:DetEqPfAll} (see Section~\ref{sec:DetEqPfAllAltProof}). For this, we provide the necessary additional definitions and supplementary facts in  Section~\ref{sec:addnot} and prove the necessary properties of the sign of a tuple of integers in Section~\ref{sec:sign}.

\section{Notation, definitions, and preliminary facts}
\label{sec:notation}

By $\F$ we denote a commutative ring. $\M_{n\, k}(\F)$ denotes the set of all $n\times k$ matrices with the entries from $\F.$  Let $E_{i\,j}\in \M_{n\,k}$ denote the matrix unit, i.e. all elements of $E_{i\,j}$ are equal to 0 except the element belonging to the intersection of the $i$-th row and $j$-th column which is equal to 1. By $\diag(a_1,\ldots, a_n) \in \M_{n\,n}(\F)$ we denote a \emph{diagonal matrix} and define it by $\diag(a_1,\ldots, a_n) = \sum_{i = 1}^{n} a_i E_{i\,i}$. By $x_{i\, j}$ we denote the element of a matrix $X$ lying on the intersection of its $i$-th row and $j$-th column. The notation $X_{i\,j}$ is also used for this purpose if it is necessary. 


For $A \in \M_{n\,k}(\F)$ by $A^t \in \M_{k\,n}(\F)$ we denote a transpose of the matrix $A$, i.e. $A^{t}_{i\,j} = A_{j\,i}$ for all $1 \le i \le k, 1 \le j \le n$.

We use the notation for submatrices following~\cite{Minc1984}. That is, given a matrix $A$, by $A[I|J]$ we denote the $|I| \times |J|$ submatrix of $A$ lying on the intersection of the rows with the indices from the set $I$ and the columns with the indices from the set $J$. By $A(I|J)$ we denote a submatrix of $A$ derived from it by striking out from it the rows with indices belonging to $I$ and the columns with the indices belonging to $J$. The absence of one of index set $I, J$ means that the corresponding set is the empty set, e.g., $A(I|)$ denotes a matrix derived from $A$ by striking out from it the rows with indices belonging to $I$. If one of $I, J$ is given explicitly, then the curly brackets used in set notation are omitted, e.g., we write $A[1,2|3,4]$ instead of $A[\{1,2\}|\{3,4\}]$. The notation with mixed brackets is also used, e.g., $A(|1]$ denotes the first column of the matrix $A$. We underline that this notation is set-based, i. e., $A[2,1|3,4] = A[1,2|3,4]$. This fact, for example, is important for the formulation and the proof of Theorem~\ref{thm:HorExpEven}. 

We use the bold font to denote vectors and lower indices to denote their coordinates. In the case if we need the series of vectors, we use the upper indices placed in braces. E.g., if $\mathbf v = \begin{pmatrix}1 \\ 0\end{pmatrix},$ then $\mathbf v^t = \begin{pmatrix}1 & 0\end{pmatrix}$ and $\mathbf v_1 = 1$. If $\mathbf u^{(1)} = \begin{pmatrix}1 \\ 0\end{pmatrix}, \mathbf u^{(2)} =  \begin{pmatrix}0 \\ 1\end{pmatrix}$, then $\mathbf u^{(1)}_1 = 1$ and $\mathbf u^{(2)}_1 = 0$.  

The following properties of the Pfaffian are used throughout this paper.

\begin{lemma}[{Cf.~\cite[Proposition~2.3(b)]{STEMBRIDGE199096}}]\label{lem:PfOfProdTrans}Let $A, B \in \M_{2m\, 2m}(\F)$ and $A$ is skew-symmetric. Then
\[
\pf(BAB^t) = \det(B) \pf (A).\]
\end{lemma}

\begin{lemma}[{Cf.~\cite[Proposition~2.3(c)]{STEMBRIDGE199096}}]\label{lem:PfD}Let $G = (g_{i\,j}) \in \M_{2m\, 2m}(\F)$ be a skew-symmetric matrix defined by $g_{i\,j} = 1$ for all $1 \le i <  j \le 2m$. Then $\pf (G) = 1$.
\end{lemma}

\begin{lemma}[{Cf.~\cite[Formula~(D1) on p.~116]{Fulton1998}}]\label{lem:PfLapExp}If $X = (x_{i\,j}) \in \M_{2m\, 2m}(\F)$ and $1 \le j \le 2m$, then
\begin{equation*}
\pf(X) = \sum_{1 \le i < j} (-1)^{i + j - 1} x_{i\,j}\pf\left(X(i, j| i, j)\right) + \sum_{j < i \le 2m} (-1)^{i + j} x_{i\,j}\pf \left(X(i, j| i, j)\right).
\end{equation*}
\end{lemma}

\section{Expressing $\det_{n\,k}$ through the Pfaffian}
\label{sec:main}

In this section we explain how Theorem~\ref{thm:DetEqPfAll} follows from Theorem~\ref{thm:OkadaPf}. The next two technical lemmas concerning the Pfaffians of the submatrices of $\D^{(n)}$ are required for this.

\begin{lemma}\label{lem:PfSubK}Let $k \le n$ and  $1 \le c_1 < \ldots < c_k \le n$  be positive integers. Then
\begin{equation}\label{lem:PfSubK:eqq1}
\D^{(n)}[c_1,\ldots, c_k|c_1,\ldots, c_k] = (-1)^{c_1 + \ldots + c_k}.
\end{equation}
\end{lemma}
\begin{proof}Indeed,
\begin{multline}\label{lem:PfSubK:eq1}
\D^{(n)}[c_1,\ldots, c_k|c_1,\ldots, c_k] = \begin{pmatrix}0 & (-1)^{c_1 + c_2} & \cdots & (-1)^{c_1+c_k}\\ \vdots & \vdots & \ddots & \vdots\\ -(-1)^{c_1 + c_{k}} & -(-1)^{c_2 + c_{k}} & \cdots & 0\end{pmatrix}\\
= \diag((-1)^{c_1}, \ldots, (-1)^{c_k}) G \diag((-1)^{c_1}, \ldots, (-1)^{c_k}),
\end{multline}
where $G = \begin{psmallmatrix}0 & 1 & \cdots & 1\\-1 & 0 & \cdots & 1\\ \vdots & \vdots & \ddots & \vdots\\ -1 & -1 & \cdots & -1\end{psmallmatrix}$ is the skew-symmetric matrix defined in Lemma~\ref{lem:PfD}. Hence,
\begin{multline*}
\pf(\D^{n}[c_1,\ldots, c_k|c_1,\ldots, c_k])\\
\overset{\eqref{lem:PfSubK:eq1}}{=\joinrel=} \pf\left(\diag((-1)^{c_1}, \ldots, (-1)^{c_k}) G \diag((-1)^{c_1}, \ldots, (-1)^{c_k})\right)\\
\overset{\mbox{\scriptsize Lemma~\ref{lem:PfOfProdTrans}}}{=\joinrel=\joinrel=\joinrel=\joinrel=\joinrel=} \det(\diag((-1)^{c_1}, \ldots, (-1)^{c_k})) \pf(G) = (-1)^{c_1 + \ldots + c_k} \pf(G)\\
 \overset{\mbox{\scriptsize Lemma~\ref{lem:PfD}}}{=\joinrel=\joinrel=\joinrel\joinrel=\joinrel=\joinrel=} (-1)^{c_1 + \ldots + c_k} \cdot 1 = (-1)^{c_1 + \ldots + c_k}.
\end{multline*}
\end{proof}

\begin{lemma}\label{lem:PfSubKP1}Let $k < n$ and  $1 \le c_1 < \ldots < c_k < n$  be positive integers. Then
\begin{equation}\label{lem:PfSubKP1:eqq1}
\D^{(n)}[1, (c_1 + 1),\ldots, (c_k + 1)|1, (c_1 + 1),\ldots, (c_k + 1)] = (-1)^{c_1 + \ldots + c_k + (k+1)}.
\end{equation}
\end{lemma}
\begin{proof}Indeed, since $c_k < n,$ then $(c_k + 1) \le n$ and by Lemma~\ref{lem:PfSubK} applied to $1 \le 1 < (c_1 + 1) < \ldots < (c_k) \le n$ we obtain that
\begin{multline*}
\D^{(n)}[1, (c_1 + 1),\ldots, (c_k + 1)|1, (c_1 + 1),\ldots, (c_k + 1)] \overset{\eqref{lem:PfSubK:eqq1}}{=\joinrel=} (-1)^{1 + (c_1+1) + \ldots + (c_k+1)}\\ = (-1)^{c_1 + \ldots + c_k + (k+1)}.
\end{multline*}

\end{proof}

\begin{proof}[Proof of Theorem~\ref{thm:DetEqPfAll}]First observe that
\begin{multline}\label{thm:DetEqPfAll:eq1}
\det_{n\,k}(X) \overset{\mbox{\scriptsize Definition~\ref{def:CullisDet}}}{=\joinrel=\joinrel=\joinrel=\joinrel=\joinrel=\joinrel=} (-1)^{(1 + \ldots + k)}\sum_{1 \le c_1 < \ldots < c_k \le n } (-1)^{c_1 + \ldots + c_k} \det X[c_1,\ldots, c_k|)\\
\overset{X \leftrightarrow X^t}{=\joinrel=\joinrel=\joinrel=} (-1)^{(1 + \ldots + k)} \sum_{1 \le c_1 < \ldots < c_k \le n } (-1)^{c_1 + \ldots + c_k} \det (X^t(|c_1,\ldots, c_k]).
\end{multline}

Now let us consider two cases separately: $k$ is even and $k$ is odd. 

\paragraph{$k$ is even} Using Lemma~\ref{lem:PfSubK} for each $1 < c_1 < \ldots < c_k \le n$ we obtain that
\begin{multline}
\det_{n\,k}(X) 
\overset{\eqref{thm:DetEqPfAll:eq1}}{=\joinrel=\joinrel=\joinrel=} (-1)^{(1 + \ldots + k)} \sum_{1 \le c_1 < \ldots < c_k \le n } (-1)^{c_1 + \ldots + c_k}\det (X^t(|c_1,\ldots, c_k])\\
\overset{\eqref{lem:PfSubK:eqq1}}{=\joinrel=} (-1)^{(1 + \ldots + k)}\\
\times  \sum_{1 \le c_1 < \ldots < c_k \le n } (-1)^{c_1 + \ldots + c_k}\pf(\D^{n}[c_1,\ldots, c_k|c_1,\ldots, c_k]) \det X^t(|c_1,\ldots, c_k])\\
\overset{\mbox{\scriptsize Lemma~\ref{thm:OkadaPf}\ref{thm:OkadaPf:part1}}}{=\joinrel=\joinrel=\joinrel=\joinrel=\joinrel=\joinrel=\joinrel=}  (-1)^{(1 + \ldots + k)}\pf \left(X^t \D^{(n)} (X^t)^t\right) = (-1)^{(1 + \ldots + k)}\pf \left(X^t \D^{(n)} X\right). 
\end{multline}

\paragraph{$k$ is odd} Recall that the statement of Theorem~\ref{thm:DetEqPfAll} for this case involves the matrix  $X' \in \M_{(n+1)\, (k+1)}(\F)$  defined by  
 $$X' =  \left(
\begin{array}{c|ccc}
\multicolumn{1}{c}{1} & 0 & \cdots & 0 \\
\cline{2-4}
0      & &   & \\
\vdots & & X & \\
0      & &   & 
\end{array}
\right).$$
We have the following expansion for $\det_{n\,k}(X)$
\begin{multline}\label{thm:DetEqPfAll:eq3}
\det_{n\,k}(X) 
\overset{\eqref{thm:DetEqPfAll:eq1}}{=\joinrel=\joinrel=\joinrel=} (-1)^{(1 + \ldots + k)} \sum_{1 \le c_1 < \ldots < c_k \le n } (-1)^{c_1 + \ldots + c_k} \det (X^t(|c_1,\ldots, c_k])\\
= (-1)^{(1 + \ldots + k + (k+1))}\sum_{1 \le c_1 < \ldots < c_k \le n} (-1)^{c_1 + \ldots + c_k + (k+1)}  \det (X^t(|c_1,\ldots, c_k]).
\end{multline}
Then, using Lemma~\ref{lem:PfSubKP1} for each $1 < c_1 < \ldots < c_k \le n$ we obtain that
\begin{multline}\label{thm:DetEqPfAll:eq4}
\det_{n\,k}(X) 
\overset{\eqref{thm:DetEqPfAll:eq3}}{=\joinrel=\joinrel=\joinrel=} (-1)^{(1 + \ldots + k + (k+1))}\\
\times \sum_{1 \le c_1 < \ldots < c_k \le n} (-1)^{c_1 + \ldots + c_k + (k+1)} \det (X^t(|c_1,\ldots, c_k])\\
\overset{\mbox{\scriptsize \eqref{lem:PfSubKP1:eqq1} for each $1 \le c_1 < \ldots < c_k \le n$ }}{=\joinrel=\joinrel=\joinrel=\joinrel=\joinrel=\joinrel=\joinrel=\joinrel=\joinrel=\joinrel=\joinrel=\joinrel=\joinrel=\joinrel=\joinrel=\joinrel=\joinrel=} (-1)^{(1 + \ldots + k + (k+1))}\\
\times \sum_{1 \le c_1 < \ldots < c_k \le n } \bigg(\pf(\D^{(n+1)}[1, (c_1+1),\ldots, (c_k+1)|1, (c_1+1),\ldots, (c_k+1)])\\
\times  \det (X^t(|c_1,\ldots, c_k])\bigg)\\
\overset{\mbox{\scriptsize Lemma~\ref{thm:OkadaPf}\ref{thm:OkadaPf:part2}}}{=\joinrel=\joinrel=\joinrel=\joinrel=\joinrel=\joinrel=\joinrel=}  (-1)^{(1 + \ldots + k + (k+1))}\pf \left((X')^t \D^{(n+1)} ((X')^t)^t\right)\\
 = (-1)^{(1 + \ldots + k + (k+1))}\pf \left((X')^t \D^{(n+1)} X'\right). 
\end{multline}
\end{proof}

\begin{remark}
The preceding proofs and definitions are division-free as well as the definition of the Pfaffian. Accordingly, Theorem~\ref{thm:DetEqPfAll} is valid for every commutative ring.
\end{remark}

\begin{example} Let us show how the polynomial equality~\eqref{thm:DetEqPfAll:eqq1} could be obtained by the direct expansion of its left-hand side and the right-hand side for $(n, k) = (3,2)$.\linebreak  Let $X = \begin{pmatrix}
x_{1\,1} & x_{1\,2}\\
x_{2\,1} & x_{2\,2}\\
x_{3\,1} & x_{3\,2}
\end{pmatrix}.$ Then
\begin{multline*}
\det_{3\,2}(X) = \begin{vmatrix}
x_{1\,1} & x_{1\,2}\\
x_{2\,1} & x_{2\,2}\\
x_{3\,1} & x_{3\,2}
\end{vmatrix}_{3\,2} = (-1)^{(1 - 1) + (2 - 2)}\begin{vmatrix}x_{1\,1} & x_{1\,2}\\ x_{2\,1} & x_{2\,2}\end{vmatrix}\\
 + (-1)^{(1 - 1) + (3 - 2)}\begin{vmatrix}x_{1\,1} & x_{1\,2}\\ x_{3\,1} & x_{3\,2}\end{vmatrix}
 + (-1)^{(2 - 1) + (3 - 2)}\begin{vmatrix}x_{2\,1} & x_{2\,2}\\ x_{3\,1} & x_{3\,2}\end{vmatrix}\\
= (x_{1\,1}x_{2\,2} - x_{1\,2}x_{2\,1}) - (x_{1\,1}x_{3\,2} - x_{1\,2}x_{3\,1}) + (x_{2\,1}x_{3\,2} - x_{2\,2}x_{3\,1})\\
= x_{1\,1}x_{2\,2} - x_{1\,1}x_{3\,2}  - x_{1\,2}x_{2\,1} + x_{1\,2}x_{3\,1} + x_{2\,1}x_{3\,2} - x_{2\,2}x_{3\,1}
\end{multline*}
Before we expand $\pf(X^t\D^{(3)}X)$, let us find $Y = X^t\D^{(3)}X.$
\begin{multline*}
Y = X^t\D^{(3)}X = \begin{pmatrix}
x_{1\,1} & x_{2\,1} & x_{3\,1}\\
x_{1\,2} & x_{2\,2} & x_{3\,2}
\end{pmatrix} \begin{pmatrix}0 & 1 & -1\\ -1 & 0 & 1\\ 1 & -1 & 0\end{pmatrix} \begin{pmatrix}
x_{1\,1} & x_{1\,2}\\
x_{2\,1} & x_{2\,2}\\
x_{3\,1} & x_{3\,2}
\end{pmatrix}\\
= \begin{pmatrix}
-x_{2\,1} + x_{3\,1} & x_{1\,1} - x_{3\,1} & -x_{1\,1} + x_{2\,1}\\
-x_{2\,2} + x_{3\,2} & x_{1\,2} - x_{3\,2} & -x_{1\,2} + x_{2\,2}
\end{pmatrix}
 \begin{pmatrix}
x_{1\,1} & x_{1\,2}\\
x_{2\,1} & x_{2\,2}\\
x_{3\,1} & x_{3\,2}
\end{pmatrix}\\
= 
{\setlength{\arraycolsep}{1pt}\scriptsize
\begin{pmatrix}
(\SM x_{2\hspace{0.01em}1} {\scriptscriptstyle +} x_{3\hspace{0.01em}1})x_{1\hspace{0.01em}1} {\scriptscriptstyle +} (x_{1\hspace{0.01em}1} \SM x_{3\hspace{0.01em}1})x_{2\hspace{0.01em}1} {\scriptscriptstyle +} (\SM x_{1\hspace{0.01em}1} {\scriptscriptstyle +} x_{2\hspace{0.01em}1})x_{3\hspace{0.01em}1} & (\SM x_{2\hspace{0.01em}1} {\scriptscriptstyle +} x_{3\hspace{0.01em}1})x_{1\hspace{0.01em}2} {\scriptscriptstyle +} (x_{1\hspace{0.01em}1} \SM x_{3\hspace{0.01em}1})x_{2\hspace{0.01em}2} {\scriptscriptstyle +} (\SM x_{1\hspace{0.01em}1} {\scriptscriptstyle +} x_{2\hspace{0.01em}1})x_{3\hspace{0.01em}2}\\
(\SM x_{2\hspace{0.01em}2} {\scriptscriptstyle +} x_{3\hspace{0.01em}2})x_{1\hspace{0.01em}1} {\scriptscriptstyle +} (x_{1\hspace{0.01em}2} \SM x_{3\hspace{0.01em}2})x_{2\hspace{0.01em}1} {\scriptscriptstyle +} (\SM x_{1\hspace{0.01em}2} {\scriptscriptstyle +} x_{2\hspace{0.01em}2})x_{3\hspace{0.01em}1} & (\SM x_{2\hspace{0.01em}2} {\scriptscriptstyle +} x_{3\hspace{0.01em}2})x_{1\hspace{0.01em}2} {\scriptscriptstyle +} (x_{1\hspace{0.01em}2} \SM x_{3\hspace{0.01em}2})x_{2\hspace{0.01em}2} {\scriptscriptstyle +} (\SM x_{1\hspace{0.01em}1} {\scriptscriptstyle +} x_{2\hspace{0.01em}1})x_{3\hspace{0.01em}2}
\end{pmatrix}}\\
{\setlength{\arraycolsep}{1pt}\scriptsize
\begin{pmatrix}
0 & x_{1\hspace{0.01em}1}x_{2\hspace{0.01em}2} \SM  x_{1\hspace{0.01em}1}x_{3\hspace{0.01em}2}  \SM  x_{1\hspace{0.01em}2}x_{2\hspace{0.01em}1} {\scriptscriptstyle +} x_{1\hspace{0.01em}2}x_{3\hspace{0.01em}1} {\scriptscriptstyle +} x_{2\hspace{0.01em}1}x_{3\hspace{0.01em}2} \SM  x_{2\hspace{0.01em}2}x_{3\hspace{0.01em}1}\\
\SM (x_{1\hspace{0.01em}1}x_{2\hspace{0.01em}2} \SM  x_{1\hspace{0.01em}1}x_{3\hspace{0.01em}2}  \SM  x_{1\hspace{0.01em}2}x_{2\hspace{0.01em}1} {\scriptscriptstyle +} x_{1\hspace{0.01em}2}x_{3\hspace{0.01em}1} {\scriptscriptstyle +} x_{2\hspace{0.01em}1}x_{3\hspace{0.01em}2} \SM  x_{2\hspace{0.01em}2}x_{3\hspace{0.01em}1}) & 0
\end{pmatrix}}.
\end{multline*}
Since the Pfaffian of the matrix of the form $\begin{psmallmatrix}0 & a\\ \SM a & 0\end{psmallmatrix}$ is $a$, then
\begin{equation*}
\pf(X^t\D^{(3)}X) = \pf(Y) = x_{1\,1}x_{2\,2} - x_{1\,1}x_{3\,2}  - x_{1\,2}x_{2\,1} + x_{1\,2}x_{3\,1} + x_{2\,1}x_{3\,2} - x_{2\,2}x_{3\,1}.
\end{equation*}
Thus, since $(-1)^{1 + 2 + 3} = 1$, the polynomials $\det_{3\,2}(X)$  $(-1)^{1 + 2 + 3}\pf(X^t\D^{(3)}X)$ are equal, which agrees with the equality~\eqref{thm:DetEqPfAll:eqq1} from the statement of Theorem~\ref{thm:DetEqPfAll}.
\end{example}

\section{Algorithm for computing $\det_{n\,k}(X)$}
\label{sec:computation}

In this section we show that Theorem~\ref{thm:DetEqPfAll} could be used to develop an algorithm for finding the value of the Cullis' determinant that improves the definition-based algorithm. 

Let us define the procedure $\mathbf{CULLISDET}(n, k, X)$ which receives as input integers $n, k \ge 1$ and a matrix $X = (x_{i\,j}) \in \M_{n\,k}(\F)$ and returns $\det_{n\,k}(X)$. This procedure calls the subprocedure $\mathbf{PF}(2n, X)$ which implements a division-free algorithm for finding the value of the Pfaffian of a given $2n\times 2n$ skew-symmetric matrix $X$ described in~\cite{PRZEZDZIECKI2025106550}. 

\begin{algorithm}[H]
\caption{The pseudocode for $\mathbf{CULLISDET}$ procedure}
\label{alg:DETNKALG}
    \begin{algorithmic}
\Procedure{$\mathbf{CULLISDET}$}{$n, k, X$}
\If{$n < k$}
\State \Return \Call{$\mathbf{CULLISDET}$}{$k, n, X^t$}
\EndIf
\If{$k$ is even}
\State $Y \gets X^t\mathbf D^{(n)}X$
\State $P \gets \Call{$\mathbf{PF}$}{k, Y}$
\State \Return $(-1)^{1 + 2 + \ldots + k}\cdot P$
\Else
\State $X' \gets \left(
\begin{array}{c|ccc}
\multicolumn{1}{c}{1} & 0 & \cdots & 0 \\
\cline{2-4}
0      & &   & \\
\vdots & & X & \\
0      & &   & 
\end{array}
\right)$
\State $Y \gets (X')^t\mathbf D^{(n+1)}(X')$
\State $P \gets \Call{$\mathbf{PF}$}{k+1,Y}$
\State \Return $(-1)^{1 + 2 + \ldots + (k+1)}\cdot P$
\EndIf
\EndProcedure
    \end{algorithmic}
\end{algorithm}

In order to study an asymptotical complexity of the presented algorithm, let $C_P(n,k)$ denote a number of arithmetical operations required to perform the procedure\linebreak $\mathbf{CULLISDET}(n, k, X)$. If $A$ and $B$ denote correspondingly an arithmetical complexity of finding the matrix $Y$ and calling the procedure $\textbf{PF}$, then
\[
C_P(n,k) = A + B + O(1).
\]

We conclude from Algorithm~\ref{alg:DETNKALG} that 
$$A = O(n^2k) + O(nk^2).$$
As it is stated in~\cite{PRZEZDZIECKI2025106550}, the subprocedure $\mathbf{PF}(2n,X)$ defined above requires $O(nM(n))$ arithmetical operations, where $M(n)$ is the arithmetical complexity of matrix multiplication of $n\times n$ matrices. Let us assume for simplicity that we use the definition-based algorithm for this operation. Since this algorithm requires $n^2(n-1)$ arithmetical operations, then $M(n) = O(n^3)$. 
Thus,
$$B = O(k\cdot k^{3}) = O(k^{4})$$
and consequently
$$C_P(n,k) = O(n^2k) + O(nk^2) + O(k^{4}) + O(1).$$
%

Now, let $C_M(n,k)$ denote a number of arithmetical operations required to calculate $\det_{n\,k}(X)$ directly by finding all the maximal minors of $X$. Then 
\[
C_M(n,k) = \binom{n}{k} D(k),
\] 
 where $D(k)$ is the number of arithmetical operations required for finding the determinant of a $k\times k$ matrix.


Let us compare the complexity of these two algorithms assuming that $n \ge k$. We consider three cases: 1) $k = 3$; 2) $k=4$ and 3) $k$ grows linearly with $n$, namely $k = \lfloor \frac{n}{2} \rfloor$. For the last case we use the inequality
\begin{equation}\label{compalg:eq1}
\binom{n}{k} \ge \left(\frac{n}{k}\right)^k
\end{equation}
(see~\cite[formula (C.5)]{Cormen2022-kd}). For $k = \lfloor \frac{n}{2} \rfloor$ the inequality~\eqref{compalg:eq1} implies that
\[
\binom{n}{\lfloor \frac{n}{2} \rfloor} \ge \left(\frac{n}{\lfloor \frac{n}{2} \rfloor}\right)^{\lfloor \frac{n}{2} \rfloor} \ge 2^{\lfloor \frac{n}{2} \rfloor}.
\]
Therefore,
\[
\binom{n}{\lfloor \frac{n}{2} \rfloor} = \Omega \left(2^{\lfloor \frac{n}{2} \rfloor}\right) = \Omega\left(\sqrt{2}^n\right)
\]
and consequently
\[
C_M(n,\lfloor \frac{n}{2} \rfloor) = \Omega\left(\sqrt{2}^n\right) D\left(2^{\lfloor \frac{n}{2} \rfloor}\right) = \Omega\left(\sqrt{2}^n\right).
\]
\begin{table}
\caption{Asymptotical behaviour of $C_M(n,k)$ and $C_P(n,k)$}
\label{tab:algcomp}
\begin{center}
\renewcommand{\arraystretch}{2}
\begin{tabular}{||c||c|c|c|}
\hline
 & $k = 3$ & $k = 4$ & $k = \lfloor \frac{n}{2} \rfloor$\\
 \hline
 \hline
$C_M(n,k)$ & $\Omega (n^3)$ & $\Omega (n^4)$ & $\Omega\left(\sqrt{2}^n\right)$\\
\hline
$C_P(n,k)$  & $O(n^2)$ & $O(n^2)$ & $O(n^3)$ \\
\hline
\end{tabular}
\end{center}
\end{table}

The results presented in Table~\ref{tab:algcomp} demonsrate that Algorithm~\ref{alg:DETNKALG} has the better asymptotical complexity than the definition-based one, especially in the last case, where it has a polynomial complexity, while the complexity of the second algorithm is exponential.
\section{The horizontal analogue of the Laplace expansion}
\label{sec:horizexp}

In this section we prove Theorem~\ref{thm:HorExpEven} and Theorem~\ref{thm:HorExpOdd} providing the formulas for the horizontal expansion of $\det_{n\,k}$ for $k$ even and $k$ odd, correspondingly. The proofs rely on Theorem~\ref{thm:DetEqPfAll}, the expansion formula for the Pfaffian (Lemma~\ref{lem:PfLapExp}) and the property of the Cullis' determinant formulated in Lemma~\ref{lem:DetNKZeroFirst} below.

%
%
\begin{proof}[Proof of Theorem~\ref{thm:HorExpEven}]Indeed,
\begin{multline*}
\det_{n\,2m}(X) \overset{\mbox{\scriptsize Theorem~\ref{thm:DetEqPfAll}}}{=\joinrel=\joinrel=\joinrel=\joinrel=\joinrel=\joinrel=} (-1)^{1 + 2 + \ldots + 2m}\cdot \pf \left(X^t \D^{(n)} X\right)\\
 \overset{\mbox{\scriptsize Lemma~\ref{lem:PfLapExp}}}{=\joinrel=\joinrel=\joinrel=\joinrel=\joinrel=} (-1)^{1 + 2 + \ldots + 2m}\left(\sum_{1 \le i < j} (-1)^{i + j - 1} \left(X^t \D^{(n)} X\right)_{i\,j}\pf\left(\left(X^t \D^{(n)} X\right)(i, j | i, j)\right)\right.\\
\phantom{XXXXXXXXXXX} \left. + \sum_{j < i \le 2m} (-1)^{i + j} \left(X^t \D^{(n)} X\right)_{i\,j}\pf\left(\left(X^t \D^{(n)} X\right)(i, j | i, j)\right)\right)\\
\overset{\mbox{\scriptsize Theorem~\ref{thm:DetEqPfAll}}}{=\joinrel=\joinrel=\joinrel=\joinrel=\joinrel=\joinrel=}  (-1)^{1 + 2 + \ldots + 2m}
 \Bigg(\sum_{1 \le i < j} \bigg((-1)^{i + j - 1} \cdot (-1)^{1 + 2} \cdot  (-1)^{1 + 2 + \ldots + (2m- 2)}\\
\phantom{XXXXXXXXXXXXXXXX}\times \det_{n\,2}\bigl(X(|i,j]\bigr) \cdot \det_{n\,(2m-2)}\bigl(X(|i,j)\bigr)\bigg)\\
\phantom{XXXXXXXXXXX} + \sum_{j < i \le 2m}\bigg( (-1)^{i + j}\cdot (-1)^{1 + 2} \cdot  (-1)^{1 + 2 + \ldots + (2m - 2)} \\
\phantom{XXXXXXXXXXXXXXXX} \times \bigl(-\det_{n\,2}\bigl(X(|i,j]\bigr)\bigr) \cdot \det_{n\,(2m-2)}\bigl(X(|i,j)\bigr)\bigg)\Bigg)\\
\overset{{\parbox{2.3cm}{\scriptsize the total exponent of $(-1)$ is even}}}{=\joinrel=\joinrel=\joinrel=\joinrel=\joinrel=\joinrel=\joinrel=\joinrel=\joinrel=\joinrel=\joinrel=} \sum_{1 \le i < j} (-1)^{i + j - 1} \cdot  \det_{n\,2}\bigl(X(|i,j]\bigr) \cdot  \det_{n\,(2m-2)}\bigl(X(|i,j)\bigr)\\
\phantom{XXXXXXXXXXX} + \sum_{j < i \le 2m} (-1)^{i + j} \cdot  \biggl(-\det_{n\,2}\bigl(X(|i,j]\bigr)\biggr)\cdot \det_{n\,(2m-2)}\bigl(X(|i,j)\bigr)\\
 = \sum_{1 \le i \neq j \le 2m} (-1)^{i + j - 1} \cdot  \det_{n\,2}\bigl(X(|i,j]\bigr)\cdot  \det_{n\,(2m-2)}\bigl(X(|i,j)\bigr)
\end{multline*}
\end{proof}

\begin{lemma}\label{lem:DetNKZeroFirst}Let $n \ge k$ be positive integers such that $k$ is even. Then 
\begin{equation}\label{lem:DetNKZeroFirst:eqq1}
\begin{vmatrix}
0 & \cdots & 0\\
x_{1\,1} & \cdots & x_{1\,k}\\
\vdots & \ddots & \vdots\\
x_{n\,1} & \cdots & x_{n\,k}
\end{vmatrix}_{(n+1)\,k} =
\begin{vmatrix}
x_{1\,1} & \cdots & x_{1\,k}\\
\vdots & \ddots & \vdots\\
x_{n\,1} & \cdots & x_{n\,k}
\end{vmatrix}_{n\,k}\;\;\mbox{for all}\;\;X = (x_{i\,j}) \in \M_{n\,k}(\F).
\end{equation}
\end{lemma}
\begin{proof}Let $Y = (y_{i\,j})$ denote the matrix on the left-hand side of~\eqref{lem:DetNKZeroFirst:eqq1} and $X = (x_{i\,j})$ denote the matrix on the right-hand side of~\eqref{lem:DetNKZeroFirst:eqq1}.
Then, by the definition of the Cullis' determinant we conclude that
\begin{multline*}
\begin{vmatrix}
0 & \cdots & 0\\
x_{1\,1} & \cdots & x_{1\,k}\\
\vdots & \ddots & \vdots\\
x_{n\,1} & \cdots & x_{n\,k}
\end{vmatrix}_{(n+1)\,k}\\
 = (-1)^{(1 + \ldots + k)}\sum_{1 \le c_1 < \ldots < c_k \le (n+1) } (-1)^{c_1 + \ldots + c_k} \begin{vmatrix}x'_{c_1\, 1} & \cdots & y_{c_1\, k}\\ \vdots & \ddots & \vdots\\ y_{c_k\, 1} & \cdots & y_{c_k\, k}\end{vmatrix}\\
\overset{\scriptsize y_{1\,j} = 0\;\;\mbox{\scriptsize  for all}\;\;1 \le j \le k}{=\joinrel=\joinrel=\joinrel=\joinrel=\joinrel=\joinrel=\joinrel=\joinrel=\joinrel=\joinrel=\joinrel=} 
(-1)^{(1 + \ldots + k)}\sum_{2 \le c_1 < \ldots < c_k \le (n+1) } (-1)^{c_1 + \ldots + c_k} \begin{vmatrix}y_{c_1\, 1} & \cdots & y_{c_1\, k}\\ \vdots & \ddots & \vdots\\ y_{c_k\, 1} & \cdots & y_{c_k\, k}\end{vmatrix}\\
(-1)^{(1 + \ldots + k)}\sum_{1 \le c'_1 < \ldots < c'_k \le n } (-1)^{(c'_1 + 1) + \ldots + (c'_k+1)} \begin{vmatrix}y_{(c'_1 + 1)\, 1} & \cdots & y_{(c'_1+1)\, k}\\ \vdots & \ddots & \vdots\\ y_{(c'_k+1)\, 1} & \cdots & y_{(c'_k+1)\, k}\end{vmatrix}\\
\overset{\parbox{3.2cm}{\scriptsize $k$ is even, $y_{i+1\,j} = x_{i\,j}$ for all $1 \le i \le n,\, 1 \le j \le k$}}{=\joinrel=\joinrel=\joinrel=\joinrel=\joinrel=\joinrel=\joinrel=\joinrel=\joinrel=\joinrel=\joinrel=\joinrel=\joinrel=\joinrel=\joinrel=}
(-1)^{(1 + \ldots + k)}\sum_{1 \le c'_1 < \ldots < c'_k \le n} (-1)^{c'_1 + \ldots + c'_k} \begin{vmatrix}x_{c'_1\,1} & \cdots & x_{c'_1\, k}\\ \vdots & \ddots & \vdots\\ x_{c'_k\, 1} & \cdots & x_{c'_k\, k}\end{vmatrix}\\
= \det_{n\,k}(X).
\end{multline*}
\end{proof}

%

\begin{proof}[Proof of Theorem~\ref{thm:HorExpOdd}] Let $X \in \M_{n\, (2m-1)}(\F)$. Consider $X' \in \M_{(n+1)\, (2m)}(\F)$  defined by  
 $$X' =  \left(
\begin{array}{c|ccc}
\multicolumn{1}{c}{1} & 0 & \cdots & 0 \\
\cline{2-4}
0      & &   & \\
\vdots & & X & \\
0      & &   & 
\end{array}
\right).$$
For all  $1 \le i \le 2m - 1$ we have
\begin{equation}\label{thm:HorExpOdd:eq1}
\det_{(n+1)\,2}\bigl(X'(|1,i+1]\bigr) = 
\begin{vmatrix}
1 & 0\\
0 & x_{1\, i}\\
\vdots & \vdots\\
0 & x_{n\, i}
\end{vmatrix}_{(n+1)\,2} \overset{\parbox{3.5cm}{\scriptsize Laplace expansion of $X'(|1,i+1]$ along the first column}}{=\joinrel=\joinrel=\joinrel=\joinrel=\joinrel=\joinrel=\joinrel=\joinrel=\joinrel=\joinrel=\joinrel=\joinrel=\joinrel=\joinrel=\joinrel=} \det_{n\,1} \bigl(X(|i]\bigr) 
\end{equation}
and
\begin{multline}\label{thm:HorExpOdd:eq2}
\det_{(n+1)\,(2m-2)}\bigl(X'(|1, i+1)\bigr)\\
 = \begin{vmatrix}
0 & \cdots & 0 & 0 & \cdots & 0 \\
x_{1\, 1} & \cdots & x_{1\, (i-1)} & x_{1\, (i+1)} & \cdots & x_{1\, (2m-1)}\\
\vdots & \ddots & \vdots & \vdots & \ddots & \vdots\\
x_{n\, 1} & \cdots & x_{n\, (i-1)} & x_{n\, (i+1)} & \cdots & x_{n\, (2m-1)}\\
\end{vmatrix}_{(n+1)\,(2m-2)}\\
\overset{\mbox{\scriptsize Lemma~\ref{lem:DetNKZeroFirst}}}{=\joinrel=\joinrel=\joinrel=\joinrel=\joinrel=} \det_{n\,(2m-2)}\bigl(X(|i)\bigr).
\end{multline}

Hence,
\begin{multline*}
\det_{n\,(2m-1)}(X) \overset{\parbox{3.5cm}{\scriptsize Laplace expansion of $X'$ along the first column}}{=\joinrel=\joinrel=\joinrel=\joinrel=\joinrel=\joinrel=\joinrel=\joinrel=\joinrel=\joinrel=\joinrel=\joinrel=\joinrel=} \det_{(n+1)\,(2m)}(X')\\
\overset{\parbox{2.6cm}{\scriptsize Theorem~\ref{thm:HorExpEven} for $X'$ and $j = 1$}}{=\joinrel=\joinrel=\joinrel=\joinrel=\joinrel=\joinrel=\joinrel=\joinrel=\joinrel=\joinrel=\joinrel=\joinrel=} \sum_{1 < i \le 2m} (-1)^{i + 1 - 1} \det_{(n+1)\,2}\bigl(X'(|1,i]\bigr)\det_{(n+1)\,(2m-2)}\bigl(X'(|1, i)\bigr)\\
\overset{\parbox{1.6cm}{\scriptsize Replacement $i = i' + 1$}}{=\joinrel=\joinrel=\joinrel=\joinrel=\joinrel=\joinrel=\joinrel=} \sum_{1 \le i' \le 2m-1} (-1)^{i' - 1} \det_{(n+1)\,2}\bigl(X'(|1,i'+1]\bigr)\det_{(n+1)\,(2m-2)}\bigl(X'(|1,i'+1)\bigr)\\
\overset{\mbox{\scriptsize \eqref{thm:HorExpOdd:eq1} and~\eqref{thm:HorExpOdd:eq2}}}{=\joinrel=\joinrel=\joinrel=\joinrel=\joinrel=\joinrel=\joinrel=\joinrel=\joinrel=} 
\sum_{1 \le i' \le 2m-1} (-1)^{i' - 1} \det_{n\,1}\bigl(X(|i']\bigr)\det_{n\,(2m-2)}\bigl(X(|i')\bigr).
\end{multline*}
\end{proof}

Using Theorem~\ref{thm:HorExpOdd}, it is possible to provide the following direct proof of the sufficiency part of Theorem 7.12 from~\cite{Guterman2025l}.

\begin{corollary}\label{cor:ZeroRowSumDetZero}Assume that $n \ge k \ge 1$ are integers and $k$ is odd. If $X \in \M_{n\,k}$ is such that the alternating sum of its rows is equal to zero, then $\det_{n\,k}(X) = 0$.
\end{corollary}
\begin{proof}Indeed, since the alternating sum of the rows of $X$ is equal to zero, then
\begin{equation}\label{cor:ZeroRowSumDetZero:eq1}
\det_{n\,1}(X(|i]) = 0\;\;\mbox{for all}\;\;1 \le i \le k.
\end{equation}
Hence,
\begin{multline*}
\det_{n\,k}(X)
 \overset{\mbox{\scriptsize Theorem~\ref{thm:HorExpOdd}}}{=\joinrel=\joinrel=\joinrel=\joinrel=\joinrel=\joinrel=\joinrel=\joinrel=} \sum_{1 \le i \le k} (-1)^{i - 1} \det_{n\,1}\left(X(|i]\right)\det_{n\,(k-1)}\left(X(| i)\right)\\
  \overset{\eqref{cor:ZeroRowSumDetZero:eq1}}{=\joinrel=\joinrel=\joinrel=} \sum_{1 \le i \le k} (-1)^{i - 1} 0\cdot \det_{n\,(k-1)}\left(X(| i)\right) = 0.
\end{multline*}
\end{proof}

\subsection*{Acknowledgements} 

We thank Soichi Okada for bringing to our attention that Theorem~\ref{thm:DetEqPfAll} is a consequence of the Ishikawa-Wakayama minor summation formula.

\subsection*{Funding information}

The research of the second author was supported by the scholarship of the Center for Absorption in Science, the Ministry for Absorption of Aliyah, the State of Israel.

\subsection*{Author contributions}

All authors have accepted responsibility for the entire content of this manuscript and
consented to its submission to the journal, reviewed all the results, and approved the final version of the
manuscript.

\subsection*{Conflict of interest}

The authors declare no conflicts of interest.

\subsection*{Data availability statement}

Not applicable.

\bibliographystyle{siamplain}
\bibliography{cullisdetpfaffian}

\appendix
\section*{Appendix}

The following three sections are devoted to an alternative elementary proof of Theorem~\ref{thm:DetEqPfAll}. We begin with the case where the $n\times k$ matrix $X$ has even number of columns, i.e., $k$ is even. In this case, we establish that the coefficients of the monomials of the form $x_{i_1\, 1} \cdot \ldots \cdot x_{i_k\, k}$ in the expansions of $\det_{n\,k}(X)$ and $\pf \left(X^t \D^{(n)} X\right)$ are equal for all $1 \le i_1, \ldots, i_k \le n$. For this, we prove that $\sgn(c_{1}, \ldots, c_{2m})$ is expressed as the Pfaffian of the $2m \times 2m$ matrix $\SGN(c_1, \ldots, c_n) = (s_{i\,j})$, where $s_{i\,j} = \sgn (c_{i}, c_{j})$ for all $1 \le i, j \le 2m$ (Definition~\ref{def:SGNMat} and Lemma~\ref{lem:PfOfSGNIsSgn}). The remaining case in the equality~\eqref{thm:DetEqPfAll:eqq1} is carried out by applying the properties of the Cullis' determinant.

\section{Additional notation and supplementary facts}
\label{sec:addnot}
We use the additional notation and supplementary facts which are introduced following~\cite{NAKAGAMI2007422}.

\begin{definition}
Let $X$ be a set and $k \ge 1$ be an integer. By $\mathcal C_{X}^{k}$ we denote the set of all injections from $[k]$ to $X$.
\end{definition}

%
%

\begin{definition}\label{def:NAKAGSgn}Given an injection $\sigma \in \mathcal C_{[n]}^{k}$ we denote by $\sgn_{n\,k} (\sigma)$ a number defined by
\[
\sgn_{n\,k} (\sigma) = (-1)^{\sum_{\alpha = 1}^{k} (\sigma(\alpha) - \alpha)} \cdot \sgn (\sigma(1),\ldots, \sigma(k)).
\]
\end{definition}

\begin{lemma}[Cf.~{\cite[Theorem 13]{NAKAGAMI2007422}}]\label{lem:CullisDetSgnDef}If  $n \ge k$ and $X = (x_{i\,j}) \in \M_{n\,k} (\F)$, then 
$$
\det_{n\,k} (X) = \sum_{\sigma \in \mathcal  C_{[n]}^{k}} \sgn_{n\,k} (\sigma) x_{\sigma(1)\,1}x_{\sigma(2)\,2}\ldots x_{\sigma(k)\, k}.
$$
\end{lemma}

\section{The sign of a tuple of integers and its properties}
\label{sec:sign}

In Section~\ref{sec:intro} we defined the notion of the sign of a tuple (see Definition~\ref{def:SignOfSec}). Here we provide and prove its necessary properties.

\begin{definition}\label{def:SGNMat}Let $(c_1, \ldots, c_n) \in \N^{n}$ be a tuple of natural numbers. Then by $\SGN(c_1, \ldots, c_n) \in \M_{n\, n}(\F)$ we denote a matrix defined by 
\[
\SGN(c_1, \ldots, c_n) = \begin{pmatrix}\sgn (c_1,c_1) & \cdots &  \sgn (c_1, c_n)\\ \vdots & \ddots & \vdots\\ \sgn (c_n, c_1) & \cdots & \sgn (c_n, c_n)\end{pmatrix}
\] 
\end{definition}

\begin{definition}\label{def:PermMat}
Let $\pi \in S_n$ be a permutation. Then by $\mathbf P^{\pi} \in \M_{n\,n}(\F)$ we denote the permutation matrix corresponding to $\pi$ which is defined by
 $\mathbf P^{\pi} = \sum_{i = 1}^n E_{i\, (\pi i)}$. 
 
The definition of the determinant of a square matrix implies that\linebreak $\det(\mathbf P^{\pi}) = \sgn(\pi(1), \ldots, \pi(n))$.
\end{definition}

\begin{lemma}\label{lem:PermMatAct}Let $\pi \in S_{n}$ be a permutation and $A = (a_{i\,j})\in \M_{n\, n}(\F)$ be a matrix. Then
\begin{equation}\label{lem:PermMatAct:eqq1}
\mathbf P^{\pi} A (\mathbf P^{\pi})^t = \begin{pmatrix}a_{\pi(1)\, \pi(1)} & \cdots & a_{\pi(1)\, \pi(n)}\\ \vdots & \ddots & \vdots\\ a_{\pi(n)\, \pi(1)} & \cdots & a_{\pi(n)\, \pi(n)} \end{pmatrix}.
\end{equation}
\end{lemma}
\begin{proof}Indeed, if $1 \le k, l \le n$, then
\begin{equation}\label{lem:PermMatAct:eq1}
(\mathbf P^{\pi} A (\mathbf P^{\pi})^t)_{k\,l} = \sum_{1 \le i, j \le n} {p^\pi}_{k\, i} a_{i\, j} {(p^\pi)}^t_{j\, l} = \sum_{1 \le i, j \le n} p^\pi_{k\, i} a_{i\, j} p^\pi_{l\, j}.
\end{equation}
Since the only nonzero element of the $i$-th row of $\mathbf P^{\pi}$ lies in the $\pi(i)$-th column and is equal to $1$, 
\begin{equation}\label{lem:PermMatAct:eq2}
\sum_{1 \le i, j \le n} p^\pi_{k\, i} a_{i\, j} p^\pi_{l\, j} = p^\pi_{k\, \pi(k)}  a_{\pi(k)\, \sigma(l)} p^\pi_{l\, \pi(l)} = a_{\pi(k)\, \pi(l)}.
\end{equation}
By combining~\eqref{lem:PermMatAct:eq1} and~\eqref{lem:PermMatAct:eq2} we obtain that
\begin{equation}\label{lem:PermMatAct:eq3}
(\mathbf P^{\pi} A (\mathbf P^{\pi})^t)_{k\,l} = a_{\pi(k)\, \pi(l)}.
\end{equation}
Since the equality~\eqref{lem:PermMatAct:eq3} holds for all  $1 \le k, l \le n$, the equality~\eqref{lem:PermMatAct:eqq1} is established.
\end{proof}

\begin{corollary}\label{cor:PermMatOnSGNMat}Let $\pi \in S_{n}$ be a permutation, and $(c_1, \ldots, c_{n}) \in \N^{n}$ be a tuple of natural numbers. Then
\[
\mathbf P^{\pi} \cdot \SGN (c_1, \ldots, c_{n}) (\mathbf P^{\pi})^t = \SGN (c_{\pi(1)}, \ldots, c_{\pi(n)})
\]
\end{corollary}

\begin{lemma}\label{lem:PermChangesSign}Let $(c_1, \ldots, c_{n}) \in \N^{n}$ be a tuple of natural numbers,  $\pi \in S_{n}$ be a permutation. Then
\[
\sgn (c_{\pi(1)}, \ldots,  c_{\pi(n)}) = \sgn(\pi(1), \ldots, \pi(n)) \cdot \sgn (c_{1}, \ldots, c_{n}).
\]
\end{lemma}
\begin{proof}Indeed, this equality clearly holds in the case if $\pi$ is a transposition and consequently holds for all $\pi \in S_{n}$.
\end{proof}

\begin{lemma}\label{lem:PfOfSGNIsSgn}Let$(c_1, \ldots, c_{2m}) \in \N^{2m}$ be a tuple of natural numbers. Then
\begin{equation}\label{lem:PfOfSGNIsSgn:eqq1}
\pf \left(\SGN(c_1, \ldots, c_{2m})\right) = \sgn ( c_1, \ldots, c_{2m}).
\end{equation}
\end{lemma}

\begin{proof}If there exist $1 \le i < j \le 2m$ such that $c_{i} = c_j$, then the equality~\eqref{lem:PfOfSGNIsSgn:eqq1} clearly holds. 

Assume now that all $c_i$ are pairwise distinct. Then there is $\pi \in S_{2m}$ such that 
\begin{equation}\label{lem:PfOfSGNIsSgn:eq3}
c_{\pi(1)} < \ldots < c_{\pi(2m)}.
\end{equation}
The equality~\eqref{lem:PfOfSGNIsSgn:eq3} implies in particular that
\begin{equation}\label{lem:PfOfSGNIsSgn:eq1}
\sgn (c_{\pi(i)}, c_{\pi(j)}) =  \sgn (i, j)\;\;\mbox{for all}\;\; 1 \le i, j \le 2m.
\end{equation}
Hence, 
\begin{equation}\label{lem:PfOfSGNIsSgn:eq2}
\mathbf P^{\pi} \SGN(c_1, \ldots, c_{2m})(\mathbf P^{\pi})^t \overset{\mbox{\scriptsize Corollary}~\ref{cor:PermMatOnSGNMat}}{=\joinrel=\joinrel=\joinrel=\joinrel=\joinrel=\joinrel=\joinrel=\joinrel=} \SGN(c_{\pi(1)}, \ldots, c_{\pi(2m)}) \overset{\eqref{lem:PfOfSGNIsSgn:eq1}}{=\joinrel=\joinrel=} G,
\end{equation}
where $G = (g_{i\,j}) \in \M_{2m\, 2m}(\F)$ is defined by $g_{i\,j} = \sgn (i, j).$ 

On the one hand,
\begin{multline}\label{lem:PfOfSGNIsSgn:eq4}
\pf\left(\mathbf P^{\pi}\SGN(c_1, \ldots, c_{2m})(\mathbf P^{\pi})^t\right) \overset{\mbox{\scriptsize Lemma~\ref{lem:PfOfProdTrans}}}{=\joinrel=\joinrel=\joinrel=\joinrel=\joinrel=} \det(\mathbf P^{\pi}) \pf \left(\SGN(c_1, \ldots, c_{2m})\right)\\
 \overset{\mbox{\scriptsize see Definition~\ref{def:PermMat}}}{=\joinrel=\joinrel=\joinrel=\joinrel=\joinrel=\joinrel=\joinrel=\joinrel=\joinrel=\joinrel=\joinrel=\joinrel=} \sgn(\pi(1), \ldots, \pi(2m)) \pf\left(\SGN(c_1, \ldots, c_{2m})\right).
\end{multline}
On the other hand,
\begin{equation}\label{lem:PfOfSGNIsSgn:eq5}
\pf\left(\mathbf P^{\pi}\SGN(c_1, \ldots, c_{2m})(\mathbf P^{\pi})^t\right) \overset{\eqref{lem:PfOfSGNIsSgn:eq2}}{=\joinrel=\joinrel=} \pf (G) \overset{\mbox{\scriptsize Lemma~\ref{lem:PfD}}}{=\joinrel=\joinrel=\joinrel=\joinrel=\joinrel=} 1. 
\end{equation}
By aligning~\eqref{lem:PfOfSGNIsSgn:eq4} and~\eqref{lem:PfOfSGNIsSgn:eq5} together we obtain that
\begin{equation}\label{lem:PfOfSGNIsSgn:eq6}
\sgn(\pi(1), \ldots, \pi(2m)) \cdot \pf\left(\SGN(c_1, \ldots, c_{2m})\right) = 1.
\end{equation}
In addition,
\begin{equation}\label{lem:PfOfSGNIsSgn:eq7}
\sgn(\pi(1), \ldots, \pi(2m)) \cdot \sgn (c_{1}, \ldots, c_{2m}) \overset{\mbox{\scriptsize Lemma~\ref{lem:PermChangesSign}}}{=\joinrel=\joinrel=\joinrel=\joinrel=\joinrel=} \sgn (c_{\pi(1)}, \ldots,  c_{\pi(2m)}) \overset{\eqref{lem:PfOfSGNIsSgn:eq3}}{=\joinrel=\joinrel=\joinrel=\joinrel=} 1
\end{equation}
By aligning~\eqref{lem:PfOfSGNIsSgn:eq6} and~\eqref{lem:PfOfSGNIsSgn:eq7} together we conclude that 
\begin{equation}\label{lem:PfOfSGNIsSgn:eq8}
\sgn(\pi(1), \ldots, \pi(2m)) \cdot \pf\left(\SGN(c_1, \ldots, c_{2m})\right) = \sgn(\pi(1), \ldots, \pi(2m)) \cdot \sgn (c_{1}, \ldots, c_{2m}).
\end{equation}
Since $\sgn(\pi(1), \ldots, \pi(2m)) = \pm 1$, the equality~\eqref{lem:PfOfSGNIsSgn:eq8} implies~\eqref{lem:PfOfSGNIsSgn:eqq1}. 
\end{proof}

\begin{lemma}\label{lem:SgnPermToSgnOfPairs}Let $(c_1, \ldots, c_{2m}) \in \N^{2m}$ be a tuple of natural numbers. Then
\begin{equation}\label{lem:SgnPermToSgnOfPairs:eqq1}
\sgn (c_1, \ldots, c_{2m}) = \sum_{\pi \in \Pi_{2m}} \sgn (\pi(1), \ldots, \pi(2m)) \Pi_{l = 1}^{m} \sgn (c_{\pi(2l-1)}, c_{\pi(2l)}).
\end{equation}
\end{lemma}

\begin{proof}
Using the definition of the Pfaffian, we obtain that
\begin{equation}\label{lem:SgnPermToSgnOfPairs:eq1}
\sum_{\pi \in \Pi_{2m}} \sgn (\pi) \Pi_{l = 1}^{m} \sgn (c_{\pi(2l-1)}, c_{\pi(2l)}) = \pf \left(\SGN (c_1, \ldots, c_{2m})\right).
\end{equation}
Then from Lemma~\ref{lem:PfOfSGNIsSgn} we obtain that 
\begin{equation}\label{lem:SgnPermToSgnOfPairs:eq2}
\pf \left(\SGN (c_1, \ldots, c_{2m})\right) = \sgn (c_1, \ldots, c_{2m}).
\end{equation} 
By aligning the equalities~\eqref{lem:SgnPermToSgnOfPairs:eq1} and~\eqref{lem:SgnPermToSgnOfPairs:eq2} together we obtain~\eqref{lem:SgnPermToSgnOfPairs:eqq1}.
\end{proof}

\section{Alternative proof of Theorem~\ref{thm:DetEqPfAll}}
\label{sec:DetEqPfAllAltProof}

In this section we provide an alternative proof of Theorem~\ref{thm:DetEqPfAll}. As in the proof in Section~\ref{sec:main} we consider separately the cases when $k$ is even and $k$ is odd. 

\begin{proof}[Proof of Theorem~\ref{thm:DetEqPfAll}]\textbf{$k$ is even.}\;\; Assume that $k = 2m$ and $X = (x_{i\,j})\in \M_{n\, 2m}(\F).$ Let us first notice that  monomials on the both sides of~\eqref{thm:DetEqPfAll:eqq1} have the form $x_{c_1\, 1} \cdot \ldots \cdot x_{c_{2m}\, 2m}$ for some tuple $(c_1, \ldots, c_{2m}) \in \mathbb N^{2m}$. Let us show that the coefficients of these monomials are equal to each other.

Lemma~\ref{lem:CullisDetSgnDef} implies that the coefficient of $x_{c_1\, 1} \cdot \ldots \cdot x_{c_{2m}\, 2m}$ on the left-hand side of~\eqref{thm:DetEqPfAll:eqq1} is equal to
\begin{equation}\label{thm:DetEqPf2K:eq3}
(-1)^{\sum_{i=1}^{2m} (c_i - i)} \sgn (c_1, \ldots, c_{2m}).
\end{equation}

To find the coefficient of $x_{c_1\, 1} \cdot \ldots \cdot x_{c_{2m}\, 2m}$ on the right-hand side of~\eqref{thm:DetEqPf2K:eqq1}, let us expand the right-hand side of~\eqref{thm:DetEqPfAll:eqq1} as follows
\begin{multline}\label{thm:DetEqPf2K:eq1}
 (-1)^{1 + 2 + \ldots + 2m}\pf \left(X^t \D^{(n)} X\right)
 = (-1)^{1 + 2 + \ldots + 2m}\\
 \times \sum_{\pi \in \Pi_{2m}} \sgn (\pi(1), \ldots, \pi(2m))  \Pi_{l = 1}^{m} \left(X^t \D^{(n)} X\right)_{\pi(2l -1)\, \pi(2l)}\\
 = (-1)^{1 + 2 + \ldots + 2m}\sum_{\pi \in \Pi_{2m}} \sgn (\pi(1), \ldots, \pi(2m)) \Pi_{l = 1}^{m} \left(\sum_{1 \le i, j \le n} x_{i\,\pi(2l-1)} d^{(n)}_{i\,j} x_{j\,\pi(2l)} \right)\\
\overset{{\scriptsize\mbox{Definition}~\ref{def:Dmat}}}{=\joinrel=\joinrel=\joinrel=\joinrel=\joinrel=\joinrel=\joinrel=} (-1)^{1 + 2 + \ldots + 2m}\phantom{XXXXXXXXXXXXXXXXXXXXXX}\\
\times \sum_{\pi \in \Pi_{2m}} \sgn (\pi(1), \ldots, \pi(2m)) \Pi_{l = 1}^{k} \left(\sum_{1 \le i, j \le n} (-1)^{i + j}\sgn (i, j) x_{i\, \pi(2l-1)} x_{j\,\pi(2l)} \right).
\end{multline}
The monomials $x_{c_1\, 1} \cdot \ldots \cdot x_{c_{2m}\, 2m}$ and $x_{i_1\, \pi(1)}x_{j_1\, \pi(2)} \cdot \ldots \cdot x_{i_1\, \pi(2m-1)}x_{j_1\, \pi(2m)}$ are equal if and only if $(i_l, j_l) = (c_{\pi(2l-1)}, c_{\pi(2l)})$ for all $1 \le l \le k$. Therefore, \eqref{thm:DetEqPf2K:eq1} implies that the coefficient of $x_{c_1\, 1} \cdot \ldots \cdot x_{c_{2m}\, 2m}$ on the right-hand side of~\eqref{thm:DetEqPfAll:eqq1} is equal to
\begin{multline}\label{thm:DetEqPf2K:eq2}
(-1)^{1 + 2 + \ldots + 2m} \left(\sum_{\pi \in \Pi_{2m}} \sgn (\pi) (-1)^{c_1 + \ldots + c_{2m}} \Pi_{l = 1}^{m} \sgn (c_{\pi(2l-1)}, c_{\pi(2l)})\right)\\
=(-1)^{\sum_{i=1}^{2m} (c_i - i)}\left(\sum_{\pi \in \Pi_{2m}} \sgn (\pi)\Pi_{l = 1}^{m} \sgn (c_{\pi(2l-1)}, c_{\pi(2l)})\right).
\end{multline}
From Lemma~\ref{lem:SgnPermToSgnOfPairs} we obtain that
\[
 \sgn (c_1, \ldots, c_{2m}) = \sum_{\pi \in \Pi_{2m}} \sgn (\pi)\Pi_{l = 1}^{m} \sgn (c_{\pi(2l-1)}, c_{\pi(2l)}).
\]
Hence, \eqref{thm:DetEqPf2K:eq3} and~\eqref{thm:DetEqPf2K:eq2} are equal. Since these coefficients are equal for all tuples\linebreak $(c_1, \ldots, c_{2m}) \in \mathbb N^{2m}$, the equality~\eqref{thm:DetEqPfAll:eqq1} is established in the considered case.

\paragraph{$k$ is odd} Recall that the statement of Theorem~\ref{thm:DetEqPfAll} for this case involves the matrix  $X' \in \M_{(n+1)\, (k+1)}(\F)$  defined by  
 $$X' =  \left(
\begin{array}{c|ccc}
\multicolumn{1}{c}{1} & 0 & \cdots & 0 \\
\cline{2-4}
0      & &   & \\
\vdots & & X & \\
0      & &   & 
\end{array}
\right).$$

Then, since $k + 1$ is even, using the definition of $X'$, we obtain that
\begin{multline*}
\det_{n\,k}(X) \overset{\parbox{3.5cm}{\scriptsize Laplace expansion of $X'$ along its first column}}{=\joinrel=\joinrel=\joinrel=\joinrel=\joinrel=\joinrel=\joinrel=\joinrel=\joinrel=\joinrel=\joinrel=\joinrel=\joinrel=\joinrel=} \det_{(n+1)\,(k+1)}(X')\\
\overset{\mbox{\scriptsize Theorem~\ref{thm:DetEqPfAll} for $X'$}}{=\joinrel=\joinrel=\joinrel=\joinrel=\joinrel=\joinrel=\joinrel=\joinrel=\joinrel=} (-1)^{1 + 2 + \ldots + (k+1)}\pf \left((X')^t \D^{(n+1)} X'\right)
\end{multline*}
\end{proof}

\end{document}